\documentclass[10pt,a4paper,reqno]{article}
\usepackage{amsmath,amssymb,amsthm}
\usepackage[paper=a4paper,left=35mm,right=35mm,top=20mm,bottom=25mm]{geometry}
\usepackage{graphicx}
\usepackage[english]{babel}
\usepackage[utf8]{inputenc}
\usepackage[T1]{fontenc}
\usepackage{mathtools}
\usepackage{bbm}
\usepackage{hyperref}
\usepackage{xcolor}
\allowdisplaybreaks

\newcommand{\p}{\varphi}
\newcommand{\N}{\mathbb{N}}
\newcommand{\Z}{\mathbb{Z}}

\newcommand{\R}{\mathbb{R}}

\newcommand{\ud}{\,\mathrm{d}}

\DeclareMathOperator{\dist}{dist}

\DeclareMathOperator{\Cov}{Cov}

\DeclareMathOperator{\PP}{P}
\DeclareMathOperator{\EE}{E}
\DeclareMathOperator{\osc}{osc}
\newcommand{\myN}{\mathcal{N}}
\newtheorem{theorem}{Theorem}[section]

\newtheorem{lemma}[theorem]{Lemma}
\newtheorem{corollary}[theorem]{Corollary}
\newtheorem{conjecture}[theorem]{Conjecture}
\theoremstyle{definition}

\theoremstyle{remark}
\newtheorem{remark}[theorem]{Remark}
\numberwithin{equation}{section}

\begin{document}

\title{Probability to be positive for the membrane model in dimensions 2 and 3}
\author{Simon Buchholz\footnote{Institut für angewandte Mathematik, Universität Bonn, Endenicher Allee 60, 53115 Bonn, Germany, \texttt{buchholz@iam.uni-bonn.de}}, Jean-Dominique Deuschel\footnote{Institut für Mathematik, Technische Universität Berlin, Straße des 17. Juni 136, 10623 Berlin, \texttt{deuschel@math.tu-berlin.de}}, Noemi Kurt\footnote{Institut für Mathematik, Technische Universität Berlin, Straße des 17. Juni 136, 10623 Berlin, \texttt{kurt@math.tu-berlin.de}}, \\ Florian Schweiger\footnote{Institut für angewandte Mathematik, Universität Bonn, Endenicher Allee 60, 53115 Bonn, Germany, \texttt{schweiger@iam.uni-bonn.de}}}
\date{\today}
\maketitle
\begin{center}
\textit{Dedicated to Amir Dembo on the occasion of his sixtieth birthday.}
\end{center}
\begin{abstract}
We consider the membrane model on a box $V_N\subset \Z^n$ of size $(2N+1)^n$ with zero boundary condition in the subcritical dimensions $n=2$ and $n=3$.
We show optimal estimates for the probability that the field is positive in a subset $D_N$ of $V_N$. In particular we obtain for $D_N=V_N$ that the probability to be positive on the entire domain is exponentially small and the rate is of the order of the surface area $N^{n-1}$.
\end{abstract}
{\footnotesize{\bf 2010 Mathematics Subject Classification.} 60G15, 60G60, 82B41\newline
}
\section{Introduction}

The membrane model  is the centred Gaussian field indexed by (a subset of) $\Z^n, n\geq 1$, whose covariance matrix is given by the Green's function of the discrete Bilaplacian. It is closely related to the well-known discrete Gaussian free field, or gradient model, whose covariance is the Green's function of the discrete Laplacian. Both of these models are considered to describe interfaces in the context of statistical physics. The particular motivation for studying the membrane model stems from physical surfaces that tend to have constant curvature, \cite{Lipowsky1995, Hiergeist1997, Ruizlorenzo2005}. The two models have many features in common.

One example is that there is a critical dimension ($n=2$ for the gradient model, $n=4$ for the membrane model), such that the variances are unbounded in the subcritical dimensions, logarithmically divergent in the critical dimension and bounded in the supercritical dimensions. See  e.g. \cite{Velenik2006} for a more general overview.

A particular feature of the gradient model is the existence of a random walk representation, which allows relatively easy estimates on the covariances, and provides proofs for correlation inequalities such as the FKG inequality. In the membrane model, such a random walk representation is present only in certain special cases, see \cite{Kurt2009}. This makes the derivation of bounds on the covariances much harder, and moreover, some widely used correlation inequalities do not hold for the membrane model. In \cite{Muller2017}, Müller and Schweiger obtained very precise estimates on the Green's function of the discrete Bilaplacian in  the subcritical dimensions 2 and 3. These results in particular imply that the membrane model is Hölder continous, \cite{Muller2017, Cipriani2018}.
Here we use the estimates to provide bounds for the probability of the interface to be positive on certain subsets of its domain. 

Such results are related to the phenomenon of entropic repulsion, which refers to the observation that some interfaces are repelled by a hard wall to a height which is determined by the fluctuations of the field. Mathematically speaking, this amounts to considering the field conditional on the event of being positive on a specified part of the domain. The field then needs to accommodate its fluctuations, so its local averages will increase. We speak of entropic repulsion if the order by which the field increases is strictly larger than the order of the square root of the variances of the original field, \cite{Giacomin2001, Lebowitz1987}. 
 
For the Gaussian gradient model entropic repulsion was proved in \cite{Bolthausen1995, Bolthausen2001, Deuschel1996, Deuschel1999}. 
For the membrane model, entropic repulsion was shown for $n\geq 4$ by Sakagawa and by Kurt \cite{Sakagawa2003, Kurt2007, Kurt2009}. In dimension $n=1$ the model corresponds to an integrated Gaussian random walk, see \cite{Caravenna2008}. Dembo, Ding and Gao \cite{Dembo2013} proved that for such processes with zero mean and finite variance the probability to be positive on an interval of lenght $N$ is of order $N^{-1/4}$, extending a result by Sinai \cite{Sinai1992} for integrated simple random walk. We consider here the membrane model defined on a box of side-length $2N+1, N\in\N$, and focus on dimensions $n=2,3$. In this case only a first result by Sakagawa \cite{Sakagawa2016} is available.

\subsection{Main results}
Let $V=[-1,1]^n$  and $V_N=NV\cap\Z^n$ with $n\in\N^+$ and $N\in\N^+$. We consider the Hamiltonian $H_N(\psi)=\frac12\sum_{x\in \Z^n}|\Delta\psi_x|^2$, where $\Delta$ is the discrete Laplacian and $\psi\in\R^{V_N}$ is a function on $V_N$, extended by 0 to all of $\Z^n$. The associated Gibbs measure
\begin{align}
	\PP_N(\mathrm{d}\psi)=\frac1{Z_N}\exp(-H_N(\psi))\prod_{x\in V_N}\ud\psi_x\prod_{x\in\Z^n\setminus V_N}\delta_0(\ud\psi_x)
\end{align}
is then the distribution of a Gaussian random field on $\Z^n$ with 0 boundary data, the so-called membrane model. We care about the subcritical case $n\in\{2,3\}$, and we are interested in the event $\Omega_{D_N,+}=\{\psi\colon\psi_x\ge0\ \forall x\in D_N\}$, where $D_N\subset V_N$, as well as the behaviour of $\psi$ conditioned on $\Omega_{D_N,+}$.

Our main result is the following.
\begin{theorem}\label{t:theoremVL}
Let $n=2$ or $n=3$. There are constants $C$, $c$ such that 
for all $N \in \N$, $L\in \N_0$ 
\begin{align}
e^{-C\frac{N^{n-1}}{(L+1)^{n-1}}}\le \PP_N(\Omega_{V_{N-L},+})\le e^{-c\frac{N^{n-1}}{(L+1)^{n-1}}}.
\end{align}
\end{theorem}
A first result in this direction was already established by Sakagawa \cite{Sakagawa2016} who proved that for every $x\in V$ there is a small neighborhood $B_x$ such that $\PP_N(\Omega_{NB_x,+})>c$ for some (universal) constant $c$.

Let us emphasize two important special cases of our theorem, which will help motivate its statement. We first consider the case $D_N=V_{\delta N}$ for $\delta\in\left(0,1\right)$, where the hard wall stays away from the boundary. In that case the fact that the membrane model is Hölder continuous suggests that the field has a decent chance to be positive if it is uniformly positive at a sufficiently dense set of lattice points of bounded cardinality. Thus the probability that $\psi$ is positive on $D_N=V_{\delta N}$ should be comparable to the probability of uniform positivity at that dense set, and thus bounded away from zero. Indeed, Theorem \ref{t:theoremVL} implies the following corollary.
\begin{corollary}\label{c:corollVdelta}
Let $n=2$ or $n=3$. For $\delta\in \left(0,1\right)$ there is a constant $c_\delta>0$ such that
\begin{align}
c_\delta\leq \PP_N(\Omega_{V_{\delta N}, +}) \leq \frac12.
\end{align}
\end{corollary}

When $D_N=V_N$, the situation is somewhat different. While the Hölder continuity holds up to the boundary, the $\psi_x$ for $x$ near the boundary are only weakly correlated and behave almost like independent random variables. This suggests that the probability to be positive on all of $V_N$ can at best scale like $e^{-cN^{n-1}}$ (note that the number of points of distance 1 to the boundary is of the order $N^{n-1}$). 
On the other hand, if the field is positive at all near-boundary points it gets pushed up in the interior quite a bit so the probability to be positive everywhere { should be of lower order.} 

Indeed, another particular case of Theorem \ref{t:theoremVL} is an estimate for $\PP_N(\Omega_{V_N,+})$.
\begin{corollary}\label{c:corollV}
Let $n=2$ or $n=3$. There are constants $C$, $c$ such that
\begin{align}e^{-CN^{n-1}}\le \PP_N(\Omega_{V_N,+})\le e^{-cN^{n-1}}.\end{align}
\end{corollary}

We expect this result to be true for the membrane model and the gradient model in any dimension $n\geq 2$. 
For the gradient model a stronger result has been shown for $n\geq 3$ in Theorem 4.1 in \cite{Deuschel1996}.
Note that the behaviour for general $L\geq 1$ in Theorem \ref{t:theoremVL} is different for the gradient model and also for the membrane model in dimension $n\geq 4$.

We give a proof of the lower and upper bound in Theorem \ref{t:theoremVL} in Section \ref{s:lowerbound} and \ref{s:upperbound}, respectively.

\subsection{Implications for entropic repulsion}
Corollary \ref{c:corollVdelta} easily implies that conditioning on $\Omega_{\delta N, +}$ does not change the order of the maximum of the field. Indeed the Hölder continuity results  from \cite{Cipriani2018} 
 (see Corollary 2.2 there) imply that $N^{-\frac{4-n}{2}} \max_{x\in V_N} \psi_x$ converges in distribution to a non-concentrated random variable $M$.
By the Borell-TIS inequality the random variables $N^{-\frac{4-n}{2}}{\max_{x\in V_N} \psi_x}{}$ have sub-Gaussian tails uniformly in $N$ and therefore
\begin{align}
0<\lim_{N\to \infty} \EE_N (N^{-\frac{4-n}{2}}{\max_{x\in V_N} \psi_x}{})
=\EE(M)<\infty.
\end{align}

Then Corollary \ref{c:corollVdelta} combined with the trivial estimate $\EE_N(X\mid \Omega_{V_{\delta N,+}})\le \frac{\EE_N(X)}{\PP_N(\Omega_{V_{\delta N,+}})}$ for   $X\geq 0$ imply
the following corollary.
\begin{corollary}
Let $n=2$ or $n=3$, and $\delta\in \left(0,1\right)$. We have that 

\begin{align}
\lim_{N\to \infty} \EE_N(N^{-\frac{4-n}{2}}\max_{x\in V_N} \psi_x \,\mid\Omega_{V_{\delta N},+})
<\frac{\EE_N(M)}{\PP_N(\Omega_{V_{\delta N,+}})}<\infty.
\end{align}

\end{corollary}
In other words conditioning on $\Omega_{V_{\delta N},+}$ changes the maximum of the field only by a bounded factor, and so there is no entropic repulsion.

We conjecture that the same holds true if we condition on $\Omega_{V_N,+}$, but a proof is more difficult since the probability of $\Omega_{V_N,+}$ is exponentially small.

In fact, we expect that conditioned on $\Omega_{V_N,+}$ a typical field looks like $\psi_x=c d_N(x)^{\frac{4-n}2}$, where $d_N(x)$ denotes the distance of $x$ to the boundary. That is, the field increases steeply near the boundary, but stays of the order $N^{\frac{4-n}2}$ in the interior.

We thus conjecture the following.
\begin{conjecture}
For $n=2,3$ we have
\begin{align}
\lim_{N\to \infty} \EE_N(N^{-\frac{4-n}{2}}\max_{x\in V_N} \psi_x \, |\, \Omega_{V_N,+})<\infty.
\end{align}

\end{conjecture}

\subsection{Notation}
Let $e_1,\ldots, e_n$ be the standard basis of $\R^n$. We use the discrete forward derivative $D_iu(x)=u(x+e_i)-u(x)$ and the discrete backward derivative $D_{-i}u(x)=u(x)-u(x-e_i)$. Then $\Delta u(x)=\sum_{i=}^nD_{-i}D_iu(x)$ denotes the discrete Laplacian.

By $\|u\|_{L^2}^2=\sum_{x\in\Z^n}u(x)^2$ we denote the $L^2$-norm of $u$, and by $(u,v)_{L^2}=\sum_{x\in\Z^n}u(x)v(x)$ the $L^2$-scalar product.

For $x\in\Z^n$ let $d_N(x)=\dist_\infty(x,\Z^n\setminus V_N)$ be the distance to the boundary of $V_N$.

For a set $A$ we denote by $|A|$ its cardinality.

In the following $c$, $C$ and $C'$ denote constants that may change from line to line, but are always independent of $N$ and $L$.
\section{Preliminaries}
Let us recall the relevant results that will be used in the proof of the main theorems. 
Let $G_N$ be the Green's function of $\Delta^2$ on $V_N$ with 0 boundary data outside $V_N$, i.e. $G_N(\cdot,y)=0$ if $y\notin V_N$ and 
\begin{align}
\begin{split}	
	\Delta^2G_N(\cdot,y)&=\delta_y \quad \text{in }V_N\\
	G_N(\cdot,y)&=0 \quad \text{outside } V_N
	\end{split}
\end{align}
if $y\in V_N$. The Greens function $G_N$ agrees with the covariance matrix of $\psi$, i.e. we have that $\Cov_N(\psi_x,\psi_y)=G_N(x,y)$, see also \cite{Kurt2009}.
Our proofs are based on the estimates for the Green's function $G_N$ recently found in  \cite{Muller2017}.
\begin{theorem}\label{t:estGN}
Let $n=2$ or $n=3$. Then we have for any $x,y\in V_N$
\begin{align}
	cd_N(x)^{4-n}&\le G_N(x,x)\le Cd_N(x)^{4-n},\label{e:estGN1}\\
	|\nabla_xG_N(x,y)|&\le Cd_N(x)^{3-n},\label{e:estGN2}\\
	|G_N(x,x)-G_N(x,y)|&\le Cd_N(x)^{3-n}|x-y|_\infty,\label{e:estGN3}\\
	|G_N(x,y)|&\le C\frac{d_N(x)^2d_N(y)^2}{(|x-y|_\infty+1)^n}.\label{e:estGN4}
\end{align}
\end{theorem}
\begin{proof}
The estimates \eqref{e:estGN1}, \eqref{e:estGN2} and \eqref{e:estGN4} are from \cite[Theorem 1.1]{Muller2017}, while \eqref{e:estGN3} follows from \eqref{e:estGN2} by discrete integration along a path from $x$ to $y$.
\end{proof}

The lower bound relies on Dudley's inequality proved in \cite{Dudley1967}.
To state the inequality we introduce the following two notions.
For a Gaussian process $(X_t)_{t\in T}$ we define the pseudometric $d_X$ by
\begin{align}\label{e:Gaussian_pseudometric}
	d_X(s,t)=\sqrt{\mathbb{E}(|X_s-X_t|^2)}.
\end{align} 
The entropy number $\myN(T,d_X,r)$ is the minimal number of open  balls of radius $r$ in the $d_X$ metric that is needed to cover $T$.

\begin{theorem}\label{t:dudley}
Let $(X_t)_{t\in T}$ be a centred Gaussian process. Then
\begin{align}
\mathbb{E}(\sup_{t\in T} X_t)\leq 24 \int_{0}^\infty \sqrt{\ln \myN(T,d_X,r)}\, \ud r.
\end{align}
\end{theorem}

\begin{remark}
The theorem is true for arbitrary sets $T$ if one defines the supremum appropriately, see e.g. \cite{Talagrand1996}. Since we only apply it to finite index sets we do not discuss this issue here.
\end{remark}

We also use the Gaussian correlation inequality due to Royen \cite{Royen2014} (see also \cite{Latala2017}).
Actually for our results the case where $K$ and $L$ are rectangles would be sufficient (see Remark \ref{re:GCI} below). In that case the theorem is due to Khatri \cite{Khatri1967} and \v Sid\'ak \cite{Sidak1967}.
\begin{theorem}\label{t:GCI}
Let $\nu$ be a centred Gaussian measure on $\R^m$ and $K,L\subset\R^m$ be closed, symmetric and convex. Then 
\begin{align}\label{e:gaussian_correlation}
\nu(K\cap L)\ge\nu(K)\nu(L).
\end{align}
\end{theorem}

Finally, we recall a Gaussian correlation inequality do to Li and Shao \cite[Lemma 5.1]{Li2004} that will be used in the proof of the upper bound
\begin{lemma}\label{l:LiShao}
Let $m\in\N$, and $X=(X_1,\ldots X_m)$, $Y=(Y_1,\ldots Y_m)$ be Gaussian random vectors with mean 0 and positive definite covariance matrices $\Sigma_X$, $\Sigma_Y$, and let $\PP$ denote their joint measure. If $\Sigma_Y\geq\Sigma_X$ (in the sense of symmetric matrices, i.e., $\Sigma_Y-\Sigma_X$ is positive semidefinite) then for every Borel set $F\subset\R^m$
\begin{align}
	\PP(Y\in F)\geq\left(\frac{\det\Sigma_X}{\det\Sigma_Y}\right)^{\frac12}P(X\in F).
\end{align}
\end{lemma}
For the convenience of the reader we repeat the short proof.
\begin{proof}
Let $f_X$, $f_Y$ be the densities of $X$ and $Y$. The assumption $\Sigma_Y\geq\Sigma_X$ implies that $\Sigma_X^{-1}\geq\Sigma_Y^{-1}$ and hence $(x,\Sigma_X^{-1}x)\geq(x,\Sigma_Y^{-1}x)$ for all $x\in\R^m$. Therefore:
\begin{align}
\begin{split}
	f_Y(x)&=\frac{1}{(2\pi)^{\frac m2}(\det\Sigma_Y)^{\frac 12}}\exp\left(-\frac12(x,\Sigma_Y^{-1}x)\right)\\
	&\geq\left(\frac{\det\Sigma_X}{\det\Sigma_Y}\right)^{\frac12}\frac{1}{(2\pi)^{\frac m2}(\det\Sigma_X)^{\frac 12}}\exp\left(-\frac12(x,\Sigma_X^{-1}x)\right)\\
	&=\left(\frac{\det\Sigma_X}{\det\Sigma_Y}\right)^{\frac12}f_X(x).
	\end{split}
\end{align}
Then
\begin{align}
\begin{split}
\PP(Y\in F)=\int_F f_Y(x)\,\ud x
&\geq\left(\frac{\det\Sigma_X}{\det\Sigma_Y}\right)^{\frac12}\int_Ff_X(x)\ud x
\\
& =\left(\frac{\det\Sigma_X}{\det\Sigma_Y}\right)^{\frac12}\PP(X\in F).
\end{split}
\end{align}
\end{proof}

\section{Lower bounds}\label{s:lowerbound}
Let \begin{align}
\Omega_{V_{N-L},\infty}:=\left\{\psi\colon|\psi_x|\le d_N(x)^{\frac{4-n}{2}}\ \forall x\in V_{N-L}\right\}
\end{align} 
be the event that $\psi$ is uniformly small on $V_{N-L}$.

If $\psi$ was $C^{0,\frac{4-n}{2}}$-Hölder continuous (with Hölder constant 1), this event  would have a positive probability uniformly in $N$ and $L$. Now $\psi$ is only $C^{0,\frac{4-n}{2}-\varepsilon}$-Hölder continuous (see \cite{Muller2017}, \cite{Cipriani2018}), so we cannot expect a lower bound independent of $N$. Instead, we prove in Subsection \ref{ss:globalsmallness} that the probability of $\Omega_{V_{N-L},\infty}$ is bounded below by $e^{-c\frac{N^{n-1}}{(L+1)^{n-1}}}$. Then, using a change of measure argument, we show in Subsection \ref{ss:changeofmeasure} that, given $f\colon V_N\to\R$, we have 

\begin{align}\PP_N(f+\Omega_{V_{N-L},\infty})\ge e^{-\frac12\|\Delta f\|_{L^2}^2}\PP_N(\Omega_{V_{N-L},\infty}).
\end{align}
Now we only need to choose $f$ such that $f(x)\ge d_N(x)^{\frac{4-n}{2}}$ for $x\in V_{N-L}$ and $\|\Delta f\|_{L^2}^2\le C\frac{N^{n-1}}{(L+1)^{n-1}}$ to prove the lower bound in Theorem \ref{t:theoremVL}.

\subsection{Local smallness of the field}\label{ss:localsmallness}
We first prove that locally the field is small with a positive probability.
For $x_0\in V_N$ and $\gamma>0$ we define the set
\begin{align}\label{e:defA}
A_{x_0,\gamma}:=\left\{x\in V_N\colon|x-x_0|_\infty\le \gamma{d_N(x_0)}\right\}.
\end{align}
\begin{lemma}\label{l:localsmall}
Let $n=2$ or $n=3$. There is a pair of constants $\gamma,\delta>0$ with the following property: For all $x_0\in V_N$ the following estimate holds 
 \begin{align}\label{e:localpositivity}
 \PP_N\left(\psi\colon|\psi_x|\le d_N(x)^{\frac{4-n}{2}}\ \forall x\in A_{x_0,\gamma}\right)\ge \delta.
 \end{align}
\end{lemma}
\begin{proof}
We apply Theorem \ref{t:dudley} to the Gaussian process $\psi$ distributed according to $\PP_N$. 
We assume $\gamma<\frac12$ so that $x\in A_{x_0,\gamma}$ implies
\begin{align}
\frac{d_N(x_0)}{2}\leq d_N(x)\leq \frac{3d_N(x_0)}{2}.
\end{align}
Therefore we will always estimate distances to the boundary for $x\in A_{x_0,\gamma}$ by $d_N(x_0)$ 
in the following. 
The bound   \eqref{e:estGN3} implies 
\begin{align}\label{e:diff_variance_bound}
\begin{split}
\mathbb{E}_N(\psi_x-\psi_y)^2&\leq |G_N(x,x)-G_N(x,y)|+|G_N(y,y)-G_N(y,x)|\\
&\leq \Theta d_N(x_0)^{3-n} |x-y|_\infty
\end{split}
\end{align}
for $x,y\in A_{x_0,\gamma}$ and some $\Theta>0$. 
Therefore we estimate the Gaussian pseudometric defined in \eqref{e:Gaussian_pseudometric}  by
\begin{align}
d_\psi(x,y)\leq \sqrt{\Theta d_N(x_0)^{3-n}|x-y|_\infty}.
\end{align}
This implies that for $x,y \in A_{x_0,\gamma}$ such that
$|x-y|_\infty\leq \frac{r^2}{\Theta d_N(x_0)^{3-n}}$ we have
\begin{align}
d_\psi(x,y) \leq r.
\end{align}
In particular $B_{d_\psi}(x,r)\subset B_\infty\left(x, \frac{r^2}{\Theta d_N(x_0)^{3-n}}\right)$
and therefore
\begin{align}
\myN(A_{x_0,\gamma}, d_\psi, r)\leq \left\lceil\frac{\gamma d_N(x_0)}{ \frac{r^2}{\Theta d_N(x_0)^{3-n}}}\right\rceil^n
\leq 1\vee \left(\frac{2\gamma\Theta d_N(x_0)^{4-n}}{ {r^2}}\right)^n.
\end{align}
Then Theorem \ref{t:dudley} implies
\begin{align}
\begin{split}\label{e:application_dudley}
\mathbb{E}_N \sup_{x\in A_{x_0,\gamma}} \psi_x&\leq
24\int_0^{\sqrt{2\gamma\Theta d_N(x_0)^{4-n}}} \sqrt{\ln \left(\frac{2\gamma\Theta d_N(x)^{4-n}}{r^2}\right)^n}\,\ud r
\\
&\leq 24d_N(x_0)^{\frac{4-n}{2}}\sqrt{2\gamma \Theta n} \int_{0}^1 \sqrt{-2\ln r}\, \ud r \leq K \sqrt{\gamma}d_N(x_0)^{\frac{4-n}{2}}
\end{split}
\end{align}
where $K$ only depends on $n$.

If we take $\gamma=(16K)^{-2}$ we obtain
\begin{align}\label{e:estE_N}
\mathbb{E}_N \left(\sup_{x\in A_{x_0,\gamma}} \psi_x\right)\le \frac{1}{16}d_N(x_0)^{\frac{4-n}{2}}
\end{align}

Define the oscillation of a function $f$ on a set $T$ as usual by
\begin{align}
\osc_{T} f=\sup_T f-\inf_T f.
\end{align}  
 Since $\psi_x$ is a centred process \eqref{e:application_dudley} implies
 \begin{align}
 \mathbb{E}_N (\osc_{A_{x_0,\gamma}} \psi_x )\leq \frac18d_N(x_0)^{\frac{4-n}{2}}. 
 \end{align}
This implies that
\begin{align}
\PP_N \left(\osc_{A_{x_0,\gamma}} \psi_x\leq \frac14 d_N(x_0)^{\frac{4-n}{2}}\right)\geq \frac12.
\end{align}
Note that we have the inclusions
\begin{align}
\begin{split}
\left\{\psi: |\psi_x|\leq d_N(x)^{\frac{4-n}{2}} \,\forall x\in A_{x_0,\gamma}\right\}
\supset \left\{\psi: |\psi_x|\leq \frac12 d_N(x_0)^{\frac{4-n}{2}} \,\forall x\in A_{x_0,\gamma}\right\}
\\
\supset 
\left\{\psi: \osc_{A_{x_0,\gamma}} \psi_x\leq \frac14 d_N(x_0)^{\frac{4-n}{2}}\right\}
\cap \left\{\psi: |\psi_{x_0}|\leq \frac14 d_N(x_0)^{\frac{4-n}{2}}\right\}. 
\end{split}
\end{align}
Now the Gaussian correlation inequality \eqref{e:gaussian_correlation}
together with \eqref{e:estGN1} imply that 
\begin{align}
\PP_N\left(\left\{\psi: |\psi_x|\leq d_N(x)^{\frac{4-n}{2}} \,\forall x\in A_{x_0,\gamma}\right\}\right)\geq \frac12 \PP_N\left( |\psi_{x_0}|\leq \frac14 d_N(x_0)^{\frac{4-n}{2}}\right)\geq \delta.
\end{align}
for some fixed $\delta>0$.
\end{proof}
\begin{remark}\label{re:GCI}
The use of the Gaussian correlation inequality could be avoided here: From \eqref{e:estE_N} and \eqref{e:estGN1} one easily obtains
\begin{align}
\mathbb{E}_N \sup_{x\in A_{x_0,\gamma}} |\psi_x|
\le \mathbb{E}_N \sup_{x\in A_{x_0,\gamma}} |\psi_x-\psi_{x_0}|
+\mathbb{E}_N |\psi_{x_0}| \le \Xi d_N(x_0)^{\frac{4-n}{2}}
\end{align}
for some $\Xi>0$
and therefore
\begin{align}
\begin{split}
 & \PP_N\left(\psi\colon|\psi_x|\le 4\Xi d_N(x)^{\frac{4-n}{2}}\ \forall x\in A_{x_0,\gamma}\right)
  \\
  &\qquad \qquad\ge \PP_N\left(\psi\colon|\psi_x|\le 2\Xi d_N(x_0)^{\frac{4-n}{2}}\ \forall x\in A_{x_0,\gamma}\right)\ge \frac12.
  \end{split}
 \end{align}
We could work with this estimate instead of \eqref{e:localpositivity} by using \begin{align}
\tilde \Omega_{V_{N-L},\infty}:=\left\{\psi\colon|\psi_x|\le 4\Xi d_N(x)^{\frac{4-n}{2}}\ \forall x\in V_{N-L}\right\}
\end{align}
instead of $\Omega_{V_{N-L},\infty}$ in the following.
\end{remark}
\subsection{Global smallness of the field}\label{ss:globalsmallness}
Using the Gaussian correlation inequality we can now conclude global estimates from Lemma \ref{l:localsmall}.

\begin{lemma}\label{l:globalV}
Let $n=2$ or $n=3$, let $\Omega_{V_{N-L},\infty}$ be as before. Then we have
\begin{align}\PP_N(\Omega_{V_{N-L},\infty})\ge e^{-C\frac{N^{n-1}}{(L+1)^{n-1}}}.\end{align}
\end{lemma}
\begin{proof}
Recall the definition of $A_{x,\gamma}$ in \eqref{e:defA}.
Fix $\gamma$ such that the conclusion of Lemma \ref{l:localsmall} holds
and use the shorter notation $A_x\coloneqq  A_{x,\gamma}$. 

We want to construct a subset $B_N$ of $V_N$ such that $|B_N|\le C\frac{N^{n-1}}{(L+1)^{n-1}}$ and such that 
\begin{align}
V_{N-L}\subset\bigcup_{x\in B_N}A_{x}.
\end{align}
If we have found such a set, then the Gaussian correlation inequality (Theorem \ref{t:GCI}) and Lemma \ref{l:localsmall} imply that
\begin{align}\begin{split} 
	\PP_N(\Omega_{V_{N-L},\infty})&=\PP_N\left(\bigcap_{x\in B_N}\{\psi\colon|\psi_y|<d_N(y)^{\frac{4-n}{2}}\ \forall y\in A_x\}\right)\\
	&\ge \prod_{x\in B_N}\PP_N(\psi\colon|\psi_y|<d_N(y)^{\frac{4-n}{2}}\ \forall y\in A_x)\\
	&\ge \prod_{x\in B_N} \delta=\delta^{|B_N|}\ge e^{-C\frac{N^{n-1}}{(L+1)^{n-1}}}.
\end{split}\end{align}
It remains to prove the existence of $B_N$. We split $V_N$ into the dyadic annuli $W_{N,k}=\{x\in V_N\colon 2^k\le d_N(x)<2^{k+1}\}$ for $k=0,1,\ldots,\lfloor\log_2N\rfloor$. For $x\in W_k$ the cube $A_x$ has diameter $2\gamma{d_N(x)}\geq \gamma 2^{k+1}$.
Because $W_k$ has outer sidelength $2(N-2^k)\le 2N$ and thickness $2^k$, we can cover it by at most
\begin{align}
2n \left(2\frac{2N}{\gamma 2^{k+1}}\right)^{n-1} 2\frac{2^k}{\gamma 2^{k+1}}
\leq C\frac{N^{n-1}}{2^{k(n-1)}}
\end{align}
cubes $A_x$, i.e. we find a set $B_{N,k}$ of at most $C\frac{N^{n-1}}{2^{k(n-1)}}$ points in $V_N$ such that
\begin{align}
W_{N,k}\subset\bigcup_{x\in B_{N,k}}A_x.
\end{align}
Let $k_{0}=\lfloor\log_2(L+1) \rfloor$ which implies  that
$V_{N-L}\subset \bigcup_{k\geq k_0} W_{N,k}$.
Consider $B_N=\bigcup_{k=k_0}^{\log_2N}B_{N,k}$. Then $V_{N-L}\subset\bigcup_{x\in B_N}A_x$, and we have
\begin{align}|A_N|\le\sum_{k=k_0}^{\lfloor\log_2N\rfloor}|A_{N,k}|\le C\sum_{k=k_0}^{\infty}\frac{N^{n-1}}{2^{k(n-1)}}\le C\frac{N^{n-1}}{2^{k_0(n-1)}}
\leq  C\frac{N^{n-1}}{(L+1)^{n-1}}.
\end{align}
\end{proof}

\subsection{Change of measure}\label{ss:changeofmeasure}
We can now prove the lower bound in Theorem \ref{t:theoremVL}. The idea is simple: We use an explicit calculation with densities to prove that the probability of the event $\PP_N(f+\Omega_{V_{N-L},\infty})$ is bounded below by $e^{-\|\Delta f\|_{L^2}^2}\PP_N(\Omega_{V_{N-L},\infty})$. Then it remains to make a good choice of $f$.
\begin{proof}[Proof of Theorem \ref{t:theoremVL}, lower bound]
Let $f\colon V_N\to\R$ be a function to be specified later, and extend it by 0 to all of $\Z^n$. We want to estimate the probability of the event $f+\Omega_{V_{N-L},\infty}=\{f+\psi:\psi\in\Omega_{V_{N-L},\infty}\}$. To do so, we calculate
\begin{align}
\begin{split}
	\PP_N&(f+\Omega_{V_{N-L},\infty})
	\\
	&=\int_{f+\Omega_{V_{N-L},\infty}}\frac1{Z_N}\exp\left(-\frac12\|\Delta\psi\|_{L^2}^2\right)\ud\psi\\
	&=\int_{\Omega_{V_{N-L},\infty}}\frac1{Z_N}\exp\left(-\frac12\|\Delta(f+\psi)\|_{L^2}^2\right)\ud\psi\\
	&=\int_{\Omega_{V_{N-L},\infty}}\frac1{Z_N}\exp\left(-\frac12\|\Delta f\|_{L^2}^2-\frac12\|\Delta \psi\|_{L^2}^2-(\Delta f,\Delta\psi)_{L^2}\right)\ud\psi.\label{e:densities1}
	\end{split}
\end{align}
Because $\Omega_{V_{N-L},\infty}$ is symmetric around the origin, we can replace $\psi$ by $-\psi$ and obtain that
\begin{align}
\begin{split}
	\PP_N&(f+\Omega_{V_{N-L},\infty})\\
	&=\int_{\Omega_{V_{N-L},\infty}}\frac1{Z_N}\exp\left(-\frac12\|\Delta f\|_{L^2}^2-\frac12\|\Delta \psi\|_{L^2}^2+(\Delta f,\Delta\psi)_{L^2}\right)\ud\psi.\label{e:densities2}
\end{split}
\end{align}
If we add \eqref{e:densities1} and \eqref{e:densities2} and use the estimate $e^t+e^{-t}\ge2$, we conclude
\begin{align}
\begin{split}
	&\PP_N(f+\Omega_{V_{N-L},\infty})\\
	&\quad=\frac12\int_{\Omega_{V_{N-L},\infty}}\frac{e^{-\frac12\|\Delta f\|_{L^2}^2-\frac12\|\Delta \psi\|_{L^2}^2}\left(e^{(\Delta f,\Delta\psi)_{L^2}}+e^{-(\Delta f,\Delta\psi)_{L^2}}\right)}{Z_N}\ud\psi\\
	&\quad\ge e^{-\frac12\|\Delta f\|_{L^2}^2}\int_{\Omega_{V_{N-L},\infty}}\frac{e^{-\frac12\|\Delta \psi\|_{L^2}^2}}{Z_N}\ud\psi\\
	&\quad=e^{-\frac12\|\Delta f\|_{L^2}^2}\PP_N(\Omega_{V_{N-L},\infty}).\label{e:estf1}
	\end{split}
\end{align}
Note that the the conclusion in \eqref{e:estf1}
also follows by Jensen from \eqref{e:densities1}.

We now choose $f$ as in Lemma \ref{l:function_construction} below. Then
\begin{equation}\label{e:estf2}
	\|\Delta f\|_{L^2}^2\le C\frac{N^{n-1}}{(L+1)^{n-1}}.
\end{equation}
Moreover this choice of $f$ ensures that $\Omega_{V_{N-L},+}\supset f+\Omega_{V_{N-L},\infty}$, and so \eqref{e:estf1}, \eqref{e:estf2} and Lemma \ref{l:globalV} imply that
\begin{align}
\begin{split} 
	\PP_N(\Omega_{V_{N-L},+})&\ge \PP_N(f+\Omega_{V_{N-L},\infty})\\
	&\ge e^{-\frac12\|\Delta f\|_{L^2}^2}\PP_N(\Omega_{V_{N-L},\infty})\\
	&\ge e^{-C\frac{N^{n-1}}{(L+1)^{n-1}}}e^{-C\frac{N^{n-1}}{(L+1)^{n-1}}}=e^{-C\frac{N^{n-1}}{(L+1)^{n-1}}}.
\end{split}\end{align}

\end{proof}

\begin{lemma}\label{l:function_construction}
There is a constant $C>0$ such that for every $N$ and $0\leq L\le N$ there is 
a function $\p:\Z^n\to \R$ such
that $\mathrm{supp}\; \p\subset V_N$, $\p(x)\geq d_N(x)^{\frac{4-n}{2}}$ for all $x\in V_{N-L}$
and 
\begin{align}
\sum_{x\in \mathbb{Z}^n} |\Delta \p(x)|^2\leq C\frac{N^{n-1}}{(L+1)^{n-1}}.
\end{align}

\end{lemma}
\begin{proof}
Recall  $W_{N,k}=\{x\in V_N\colon 2^k\le d_N(x)<2^{k+1}\}$ for $k=0,1,\ldots,\lfloor\log_2N\rfloor$. Let in addition $W_{N,-1}=\Z^n\setminus V_N$.

Fix a smooth function $\eta\colon\R\to\R$ such that $\eta\ge0$, $\eta=1$ on $[1,\infty)$ and $\eta=0$ on $(-\infty,0]$. For $i\in\{1,2,\ldots,n\}$ and $x\in\Z^n$ we introduce the distance  $d_i(x)=\dist(x,\Z^n\setminus(\Z^{i-1}\times\{-N,\ldots,N\}\times\Z^{n-i}))$ of $x$ to the boundary in direction $x_i$.

For $j=0,1,\ldots \lfloor\log_2N\rfloor-1$ consider the function
\begin{align}
	\p_j(x)=2^{\frac{j(4-n)}{2}+1} \prod_{i=1}^n\eta\left(\frac{d_i(x)}{2^j}\right).
\end{align}
Note that 
\begin{align}\label{e:phiL_lower_bound}
\p_j(x)= 2^{\frac{j(4-n)}{2}+1}
\end{align} for all $x\in V_N$ such that $d_N(x)\geq 2^j$.
Moreover 
\begin{align}\label{e:laplace_estimate_phiL}
|\Delta \p_j(x)|\le C2^{\frac{j(4-n)}{2}+1}\| \eta''\|_{\infty}\frac{1}{2^{2j}}\le\frac{ C \lVert \eta''\rVert_{\infty} }{ 2^{\frac{jn}{2}} }.
\end{align}
In fact $\Delta \p_j(x)=0$ if $d_N(x)>2^j$ because
$\p_k$
is constant on $V_{N-2^j}$.
We define the function
\begin{align}
\p=\sum_{j=\lfloor\log_2(L+1)\rfloor}^{\lfloor\log_2N\rfloor} \p_j.
\end{align}
For $x\in V_{N-L}$ let now $k$ be such that $x\in W_{N,k}$, and observe that $\lfloor\log_2(L+1)\rfloor\le k\le\lfloor\log_2N\rfloor$.
The estimate \eqref{e:phiL_lower_bound} implies
\begin{align}
\p(x)\ge \p_k(x)\ge 2^{\frac{k(4-n)}{2}+1}\ge\left(2\cdot2^k\right)^{\frac{4-n}{2}}\ge d_N(x)^{\frac{4-n}{2}}.  
\end{align}
For an arbitrary $x\in \Z^n$ let again $k\in\{-1,0,1,\ldots\}$ be such that $x\in W_{N,k}$. Then \eqref{e:laplace_estimate_phiL} implies that
\begin{align}
|\Delta \p(x)|
\le \sum_{j=k\vee \lfloor\log_2(L+1)\rfloor}^{\lfloor\log_2N\rfloor} |\Delta \p_j|
\le \sum_{j=k\vee \lfloor\log_2(L+1)\rfloor}^\infty \frac{ C \|\eta''\|_{\infty} }{ 2^{\frac{jn}{2}} }
\le \frac{C'}{2^{\frac{(k\vee \lfloor\log_2(L+1)\rfloor)n}{2}}}.
\end{align}

Using that $|W_{N,k}|\leq C2^kN^{n-1}$ for $k\ge0$ and that on $W_{N,-1}$ $\Delta \p(x)$ is zero except possibly on the set $V_{N+1}\setminus V_N$ of cardinality $CN^{n-1}\le C'2^{-1}N^{n-1}$, the previous estimate implies that

\begin{align}\begin{split} 
\sum_{x\in \Z^n} |\Delta \p(x)|^2
&\le \sum_{k=-1}^{\lfloor \log_2N\rfloor}\sum_{x\in W_{N,k}}|\Delta \p(x)|^2\\
&\le \sum_{k=-1}^{\infty} \frac{C2^kN^{n-1}}{2^{(k\vee \lfloor\log_2(L+1)\rfloor)n}}\\
&\le \sum_{k=-1}^{\lfloor\log_2(L+1)\rfloor} \frac{C2^kN^{n-1}}{2^{\lfloor\log_2(L+1)\rfloor n}}+ \sum_{k=\lfloor\log_2(L+1)\rfloor+1}^\infty\frac{C2^kN^{n-1}}{2^{kn}}\\
&\le C\frac{N^{n-1}}{(L+1)^{n-1}}+C\frac{N^{n-1}}{(L+1)^{n-1}}=C'\frac{N^{n-1}}{(L+1)^{n-1}}.
\end{split}\end{align}
\end{proof}

\section{Upper bounds}\label{s:upperbound}
In order to prove the upper bound in Theorem \ref{t:theoremVL}, we will find a suitably sparse set $E_{N,L}$ of points at the boundary such that the $\{\psi_x\colon x\in E_{N,L}\}$ are almost independent in the sense that their covariance matrix is diagonally dominant. We can then use Lemma \ref{l:LiShao} to compare them to actually independent random variables. The following argument is taken from \cite[Section 6.2.1]{Schweiger2016}

\begin{proof}[Proof of Theorem \ref{t:theoremVL}, upper bound]
Note that for $L>\frac{N}{2}$ the upper bound is trivial so we can assume $L\leq \frac{N}{2}$.
Let $\alpha>0$ be a constant to be chosen later, and let $E_{N,L}=V_{N-L}\cap ((\lceil\alpha L\rceil\Z)^{n-1}\times\{ N-L\})$. This is a set of points on one face of $[-N+L,N-L]^n$ such that any two points have distance at least $\alpha L$. Its cardinality satisfies
\begin{align}\label{e:size_independent_subset}
|E_{N,L}|=\left(2\left\lfloor \frac{N-L}{\lceil(\alpha (L+1)\rceil}\right\rfloor+1\right)^{n-1}\geq c\frac{N^{n-1}}{\alpha^{n-1}(L+1)^{n-1}}
\end{align}
Clearly $d_N(x)=d_N(y)=L+1$  for any $x, y \in E_{N,L}$  and therefore according to \eqref{e:estGN4} for $x\neq y$ 
\begin{align}
|G_N(x,y)|\leq C\frac{(L+1)^4}{(|x-y|_\infty+1)^n}\le  C\frac{(L+1)^4}{|x-y|_\infty^n}
\end{align}
If we combine this with \eqref{e:estGN1} we obtain for any $x\in E_{N,L}$
\begin{align}
\begin{split}
	\sum_{\substack{y\in E_{N,L}\\y\neq x}}&\frac{|G_N(x,y)|}{\sqrt{G_N(x,x)G_N(y,y)}}\\
	&\le C\sum_{\substack{y\in E_{N,L}\\y\neq x}}\frac{(L+1)^4}{(L+1)^{4-n}|x-y|_\infty^n}\\
	&= C\sum_{j=1}^\infty|\{y\in E_{N,L}\colon |y-x|_\infty=j\lceil\alpha (L+1)\rceil\}|\frac{(L+1)^n}{(j\lceil\alpha (L+1)\rceil)^n}\\
	&\le \frac{C}{\alpha^n}\sum_{j=1}^\infty \frac{a_j}{j^n}
	\end{split}
\end{align}
where $a_j= \begin{cases}2&\text{for $n=2$}\\4j+4&\text{for $n=3$}\end{cases}$. Thus $\sum_{j=1}^\infty \frac{a_j}{j^n}<\infty$ and hence
\begin{equation}\label{e:sumcorrel}
	\sum_{\substack{y\in E_{N,L}\\y\neq x}}\frac{|G_N(x,y)|}{\sqrt{G_N(x,x)G_N(y,y)}}\le \frac{C}{\alpha^n}.
\end{equation}
We now choose $\alpha$ large enough that the right hand side of \eqref{e:sumcorrel} becomes less than $\frac14$.

We 
define the Gaussian random vector $(X_x)_{x\in E_{N,L}}$  by $X_x=\frac{\psi_x}{\sqrt{G_N(x,x)}}$. Let $\Sigma_X$ be its covariance matrix. Then $(\Sigma_X)_{x,x}=1$ for all $x$ and \eqref{e:sumcorrel} implies that
\begin{equation}\label{e:sumcorrelX}
	\sum_{\substack{y\in E_{N,L}\\y\neq x}}|(\Sigma_X)_{x,y}|\le\frac14.
\end{equation}
Let $\{Y_x\}_{x\in E_{N,L}}$ be i.i.d. normal variables distributed according to $\mathcal{N}\left(0,\frac32\right)$ and let $\Sigma_Y=\frac32 \mathbbm{1}_{|E_{N,L}|}$ be their joint covariance matrix.

Because of equation \ref{e:sumcorrelX} the matrix $\Sigma_Y-\Sigma_X$ then satisfies
\begin{align}(\Sigma_Y-\Sigma_X)_{x,x}=\frac32-1=\frac12>\sum_{\substack{y\in E_{N,L}\\y\neq x}}(\Sigma_X)_{x,y}\end{align}
This means that $\Sigma_Y-\Sigma_X$ is strictly diagonally dominant and hence positive definite. Hence we can apply Lemma \ref{l:LiShao} and obtain
\begin{align}
\begin{split}
	\left(\frac{1}{2}\right)^{|E_{N,L}|}&=\PP(Y\in (0,\infty)^{|E_{N,L}|})\\
	&\geq\left(\frac{\det\Sigma_X}{\det\Sigma_Y}\right)^{\frac12}\PP(X\in (0,\infty)^{|E_{N,L}|})\\
	&=\left(\frac{\det\Sigma_X}{\det\Sigma_Y}\right)^{\frac12}\PP_N(\psi_x\geq 0\ \forall x\in E_{N,L})\\
	&\ge\left(\frac{\det\Sigma_X}{\det\Sigma_Y}\right)^{\frac12}\PP_N(\Omega_{V_N,+}).
	\end{split}
\end{align}

It remains to estimate $\frac{\det\Sigma_X}{\det\Sigma_Y}$. Since $\Sigma_Y$ is diagonal, $\det\Sigma_Y=\left(\frac{3}{2}\right)^{|E_{N,L}|}$.

On the other hand, by \eqref{e:sumcorrelX} the matrix $\Sigma_X-\frac34\mathbbm{1}_{|E_{N,L}|}$ is still diagonally dominant and hence positive semidefinite. Hence all eigenvalues of $\Sigma_X$ must be at least $\frac34$. Therefore $\det\Sigma_X\geq\left(\frac{3}{4}\right)^{|E_{N,L}|}$.

We conclude
\begin{align}
\begin{split}
\PP_N(\Omega_{V_N,+})\leq\left(\frac{1}{2}\right)^{|E_{N,L}|}\left(\frac{\det\Sigma_Y}{\det\Sigma_X}\right)^{\frac12}
&\leq\left(\frac{1}{2}\right)^{|E_{N,L}|}\left(\frac{3/2}{3/4}\right)^\frac{|E_{N,L}|}2
\\
&
=\left(\frac{1}{\sqrt{2}}\right)^{|E_{N,L}|}.
\end{split}\end{align}
If we recall that by \eqref{e:size_independent_subset} $|E_{N,L}|\geq c\frac{N^{n-1}}{\alpha^{n-1}(L+1)^{n-1}}$, we finally obtain
\begin{align}\PP_N(\Omega_{V_{N-L},+})\leq\exp\left(-c\frac{N^{n-1}}{(L+1)^{n-1}}\right)\end{align}
for $c=\frac1{2\alpha^{n-1}}\log2$.
\end{proof}

\bigskip
\noindent
{\bf Acknowledgements}\\
The authors would like to thank Stefan Müller for several helpful discussions. Jean-Dominique Deuschel would like to thank Amir Dembo and Jason Miller for valuable comments.\\
Simon Buchholz and Florian Schweiger were partially supported by the German Research Foundation through the Collaborative
Research Centre 1060 {\it The Mathematics of Emergent Effects}.
Simon Buchholz was suppported by the {\it Bonn International Graduate School in Mathematics (BIGS)}. Florian Schweiger was supported by the {\it Studienstiftung des deutschen Volkes}.

\bibliographystyle{alpha}
\bibliography{EntropicRepulsion}
\end{document}